\documentclass[12pt]{amsart}
\usepackage{amsmath,amsthm,amssymb,amscd,url,enumerate,mathtools}
\usepackage{graphicx}
\usepackage[margin=1in]{geometry}
\usepackage{afterpage}
\usepackage{tikz}
\usepackage{todonotes}
\usepackage[all]{xy}
\usepackage[colorlinks=true,citecolor={blue},linkcolor = {purple}, backref=page]{hyperref} 

\renewcommand*{\backrefalt}[4]{
    \ifcase #1 Not cited.%
          \or Cited on page~#2.%
          \else Cited on pages #2.%
    \fi} 
    
\usepackage{tabularx}

\usepackage{caption}

\usepackage{dutchcal}

\usepackage[square]{natbib}
\setcitestyle{numbers}

\newcommand{\TITLE}{Iterates of Quadratics and Monogenicity}
\newcommand{\TITLERUNNING}{}

\theoremstyle{plain}
\newtheorem{theorem}{Theorem}

\newtheorem{proposition}[theorem]{Proposition}
\newtheorem{lemma}[theorem]{Lemma}
\newtheorem{corollary}[theorem]{Corollary}

\theoremstyle{definition}

\theoremstyle{remark}
\newtheorem{remark}[theorem]{Remark}

\newtheorem{example}[theorem]{Example}
\newtheorem{question}[theorem]{Question}

\numberwithin{theorem}{section}

  {\end{list}}

  {\end{list}}

\newcommand{\tightoverset}[2]{%
  \mathop{#2}\limits^{\vbox to -.5ex{\kern-1.05ex\hbox{$#1$}\vss}}}

\numberwithin{equation}{section} 

\newcommand{\gl}{{\mathfrak{l}}}
\newcommand{\gm}{{\mathfrak{m}}}
\newcommand{\gn}{{\mathfrak{n}}}
\newcommand{\gp}{{\mathfrak{p}}}

\newcommand{\gP}{{\mathfrak{P}}}

\def\Ocal{{\mathcal O}}

\newcommand{\FF}{\mathbb{F}}

\newcommand{\QQ}{\mathbb{Q}}

\newcommand{\ZZ}{\mathbb{Z}}

\newcommand{\dnd}{\nmid}

\newcommand{\ol}[1]{\overline{#1}}

\newcommand{\Disc}{\operatorname{Disc}}
\newcommand{\ind}{\operatorname{ind}}

\newcommand{\Norm}{\operatorname{Norm}}

\newcommand{\red}{\operatorname{red}}

\title[\TITLERUNNING]{\TITLE}

\author[Hanson Smith]{Hanson Smith}
\address{Department of Mathematics\\
California State University San Marcos\\
333 S. Twin Oaks Valley Rd.\\
San Marcos, CA 92096}
\email{hsmith@csusm.edu}

\author[Zack Wolske]{Zack Wolske}
\address{Toronto, Ontario}
\email{zackwolske@gmail.com}


\begin{document}

\sloppy 




\begin{abstract}
We investigate monogenicity and prime splitting in extensions generated by roots of iterated quadratic polynomials. Let $f(x)\in\mathbb{Z}[x]$ be an irreducible, monic, quadratic polynomial, and write $f^n(x)$ for the $n^{\text{th}}$ iterate. We obtain necessary and sufficient conditions for $f^n(x)$ to be monogenic for each $n$. We use this to construct multiple families where $f^n(x)$ is monogenic for every $n>0$.  
\end{abstract}

\maketitle


\section{Introduction}
A number field $K$, with ring of integers $\Ocal_K$ is \emph{monogenic} if $\Ocal_K=\ZZ[\alpha]$ for some algebraic integer $\alpha$. In this case the minimal polynomial of $\alpha$ is also called \textit{monogenic}. 
If $p\nmid\big[\Ocal_K:\ZZ[\alpha]\big]$, then we say the extension is \textit{$p$-maximal}. When the compositions $f^n(x)$ of a monic integer polynomial $f(x)$ with itself are all irreducible, they produce a tower of nested fields $\QQ \subset K_1 \subset K_2 \subset \cdots$, with $K_i=\QQ(\alpha_i)$ where $f(\alpha_1)=0$ and $f(\alpha_{i+1})=\alpha_i$. When each $f^i(x)$ is monogenic, we say that $f(x)$ is \emph{dynamically monogenic}. 

Previous work (\cite{AitkenIterated}, \cite{GassertChebyshev}, \cite{Castillo}, \cite{Ruofan}) has investigated dynamical monogenicity for certain families of quadratic polynomials and Chebyshev polynomials of prime degree. 
If $f(x)=x^p-a$ is irreducible for prime $p$ and integer $a$, the first author gave necessary and sufficient conditions for $f(x)$ to be dynamically monogenic \cite{SmithDynRadical}. In this paper, we give necessary and sufficient conditions for a quadratic polynomial to be dynamically monogenic: 

\begin{theorem}\label{Thm: Main}
    Let $f(x)=x^2+bx+c$ be an irreducible polynomial in $\ZZ[x]$. Let $\alpha_n$ be a root of the $n$-fold composition 
    \[f^n =\underbrace{f\circ \cdots \circ f}_{n\text{-times}},\]
    and write $K_n=\QQ(\alpha_n)$. 
    The polynomial $f(x)$ is 2-maximal if and only if the composition $f^n(x)$ is 2-maximal for all $n$ if and only if one of the following hold:
    \begin{enumerate}
    \item[$\bullet$] $b$ is odd, 
    \item[$\bullet$] $b$ is even and $c\equiv 2 \bmod 4$,
    \item[$\bullet$] $b$ is even and $b+c\equiv 1 \bmod 4$. 
    \end{enumerate}
    In the first case (when $b$ is odd), 2 is unramified. In the latter two cases, 2 is totally ramified.
    Fix $N>0$ and suppose $f^n(x)$ is irreducible for $1\leq n\leq N$. For $k<N$, write $f_{N/k}(x)$ for the minimal polynomial of $\alpha_N$ over $K_k$. Let $p$ be an odd prime. The polynomial $f^n(x)$ is $p$-maximal for each $0<n\leq N$ if and only if 
        \[ p^2 \text{ does not divide } 4^{2^{n}}f^{n}\left(-\frac{b}{2}\right)=4^{2^{n-1}}f^{n-1}\left(-\frac{\Disc(f)}{4}\right)=\Norm_{K_n/\QQ}\left(\Disc(f_{n/n-1})\right)\] for each $0<n\leq N$. Here $f^0(x)$ is the identity function.
\end{theorem}

In Section \ref{Sec: Background}, we describe Ore's theorem and the Montes algorithm - the background results that we use to determine $p$-maximality of a quadratic extension of an arbitrary number field. Section \ref{Sec: Main} contains the proofs of the necessary and sufficient conditions in the main result and their corollaries. The major components are conditions for a relative quadratic extension to be monogenic and algebraic arguments that make $p$-maximality or non-maximality extend through the tower of fields. In Section \ref{Sec: Factorization}, we describe the splitting behaviour of the ideal $(2)$ in every case where the extension is $2$-maximal. Section \ref{Sec: Examples} contains three one-parameter families of quadratic polynomials that are each dynamically monogenic for infinitely many choices of the parameter. These arise from the connection to post-critically finite polynomials, where the set of primes to check for $p$-maximality is finite. 
A generalization to arbitrary polynomials from the perspective of critical points is forthcoming.

\section{Background: The Montes Algorithm and Ore's Theorems}\label{Sec: Background}

The Montes algorithm is a $p$-adic factorization algorithm that extends the pioneering work of {\O}ystein Ore \cite{Ore}. Though we use the notation and setup of the modern implementation, we will only make use of the aspects developed by Ore. The following is a brief summary of the tools we use; for a complete development of the Montes algorithm, see \cite{GMN}. Our notation will roughly follow \cite{ElFadilMontesNart}, which also provides a more extensive summary.

Let $p$ be an integral prime, $K$ a number field with ring of integers $\Ocal_K$, and $\gp$ a prime of $K$ above $p$. Write $K_\gp$ to denote the completion of $K$ at $\gp$. Suppose $f(x)\in \Ocal_K[x]$ is an irreducible polynomial. We extend the standard $\gp$-adic valuation to $\Ocal_K[x]$ by defining  
	\[ v_\gp\big(a_n x^n + \cdots + a_1 x + a_0\big) = \min_{0 \leq i \leq n} \big( v_\gp(a_i) \big). \]
If $\phi(x), f(x) \in \Ocal_K[x]$ are such that $\deg \phi \leq \deg f$, then we can write
    	\[f(x)=\sum_{i=0}^m a_i(x)\phi(x)^i,\]
for some $m$, where each $a_i(x) \in \Ocal_K[x]$ has degree less than $\deg \phi$. We call the above expression the \emph{$\phi$-adic development} of $f(x)$. We associate to the $\phi$-adic development of $f(x)$ an open Newton polygon by taking the lower convex hull 
of the integer lattice points $\big(i,v_p(a_i(x))\big)$. The sides of the Newton polygon with negative slope are the \emph{principal $\phi$-polygon}. The number of positive integer lattice points on or under the principal $\phi$-polygon is called the \emph{$\phi$-index of $f$} and denoted $\ind_\phi(f)$. This index is directly connected to $p$-maximality. 

Write $k_\gp$ for the finite field $\Ocal_K/\gp$, and let $\ol{f(x)}$ be the image of $f(x)$ in $k_\gp[x]$. We will develop $f(x)$ with respect to $\phi(x)$, a monic lift of an irreducible factor $\ol{\phi(x)}$ of $\ol{f(x)}$ in $k_\gp[x]$. In this situation, we will want to consider the extension of $k_\gp$ obtained by adjoining a root of $\ol{\phi(x)}$. We denote this finite field by $k_{\gp,\phi}$. We associate to each side of the principal $\phi$-polygon a polynomial in $k_{\gp,\phi}[y]$. Suppose $S$ is a side of the principal $\phi$-polygon with initial vertex $\big(s,v_\gp(a_s(x))\big)$, terminal vertex $\big(k,v_\gp(a_k(x))\big)$, and slope $-\frac{h}{e}$ written in lowest terms. Define the \emph{length} of the side to be $l(S)\coloneqq k-s$ and the \emph{degree} to be $d\coloneqq\frac{l(S)}{e}$. Let $\red:\Ocal_K[x]\to k_{\gp,\phi}$ denote the homomorphism obtained by quotienting by the ideal $\big(\gp,\phi(x)\big)$.
For each $j$ in the range $0\leq j\leq d$, we set $i=s+je$ and define the \emph{residual coefficient} as
\[c_i\coloneqq\left\{
\begin{array}{cl}
0 &\text{ if }  \big(i,v_\gp(a_i(x))\big)  \text{ lies strictly above } S  \text{ or } v_\gp(a_i(x))=\infty,\\
\red\left(a_i(x)/\pi^{v_\gp(a_i(x))}\right)  &\text{ if }  \big(i,v_\gp(a_i(x))\big) \text{ lies on } S.
\end{array}
\right.\]
Finally, the \emph{residual polynomial} of the side $S$ is the polynomial
\[R_S(y)=c_s+c_{s+e}y+\cdots +c_{s+(d-1)e}y^{d-1}+c_{s+de}y^d\in k_{\gp,\phi}[y].\]
Notice, that $c_s$ and $c_{s+de}$ are always nonzero since they correspond to the initial and terminal vertices, respectively, of the side $S$.

With all of these definitions in hand, we state two theorems that encapsulate how we will employ the Montes algorithm. The first focuses on the indices of monogenic orders.

\begin{theorem}[Ore's theorem of the index]\label{Thmofindex}
Let $f(x)\in \Ocal_K[x]$ be a monic irreducible polynomial and let $\alpha$ be a root. Choose monic polynomials $\phi_1,\dots, \phi_s \in \Ocal_K[x]$ whose reductions modulo $\gp$ are exactly the distinct irreducible factors of $\overline{f(x)} \in k_\gp[x]$. Then, 
\[v_p\left(\big[\Ocal_{K(\alpha)}:\Ocal_K[\alpha]\big]\right)\geq \ind_{\phi_1}(f)+\cdots + \ind_{\phi_s}(f).\]
Further, equality holds if, for every $\phi_i$, each side of the principal $\phi_i$-polygon has a separable residual polynomial.
\end{theorem}
For our applications, we will employ the following equivalence:

\begin{corollary}\label{iffCor}
The prime $\gp$ does not divide $\big[\Ocal_{K(\alpha)}:\Ocal_K[\alpha]\big]$ if and only if $\ind_{\phi_i}(f)=0$ for all $i$. In this case each principal $\phi_i$-polygon is one-sided.
\end{corollary}


The second theorem we state connects prime splitting and polynomial factorization. The ``three dissections" that we will outline below are due to Ore, and the full Montes algorithm is an extension of this. Our statement loosely follows Theorem 1.7 of \cite{ElFadilMontesNart}.

\begin{theorem}\label{Thm: Ore}[Ore's Three Dissections]
Let $f(x)\in \Ocal_K[x]$ be a monic irreducible polynomial and let $\alpha$ be a root. Suppose
\[\ol{f(x)}=\ol{\phi_1(x)}^{\ r_1}\cdots \ol{\phi_s(x)}^{\ r_s}\]
is a factorization into irreducibles in $k_\gp[x]$. Hensel's lemma shows $\ol{\phi_i(x)}^{\ r_i}$ corresponds to a factor of $f(x)$ in $K_\gp[x]$ and hence to a factor $\gm_i$ of $\gp$ in $K(\alpha)$. 

Choosing a lift $\phi_i(x)\in \Ocal_K[x]$ of $\ol{\phi_i(x)}\in k_\gp[x]$, suppose the principal $\phi_i$-polygon has sides $S_1,\dots, S_g$. Each side of this polygon corresponds to a distinct factor of $\gm_i$. 

Write $\gn_j$ for the factor of $\gm_i$ corresponding to the side $S_j$. Suppose $S_j$ has slope $-\frac{h}{e}$. If the residual polynomial $R_{S_j}(y)$ is separable, then the prime factorization of $\gn_j$ corresponds to the factorization of $R_{S_j}(y)$ in $k_{\gp,\phi_i}[y]$, but every factor of $R_{S_j}(y)$ will have an exponent of $e$. In other words,
\[\text{if } R_{S_j}(y)=\gamma_1(y)\dots\gamma_k(y) \ \text{ in }  \ k_{\gp,\phi_i}[y], \ \text{ then } \ \gn_j=\gP_1^{e}\cdots \gP_k^{e} \ \text{ in } \ K(\alpha)\]
with $\deg(\gamma_m)$ equaling the residue class degree of $\gP_m$ for each $1\leq m\leq k$.
In the case where $R_{S_j}(y)$ is not separable, further developments are required to factor $\gp$. 
\end{theorem}



\section{Criteria for the Monogenicity of a Quadratic}\label{Sec: Main}
Let $K$ be an arbitrary number field. Let $f(x)=x^2+bx+c\in \Ocal_K[x]$ be irreducible and suppose $\alpha$ is a root. The discriminant is $\Disc(f)=b^2-4c$. If $\gp$ is a prime of $K$ such that $v_\gp\big(\Disc(f)\big)<2$, then $\Ocal_K[\alpha]$ is $\gp$-maximal. Thus suppose $\gp\nmid 2$ and $\gp^2\mid \Disc(f)$. Hence $b^2\equiv 4c\bmod \gp^2$. Let $t\in \Ocal_K$ be such that $2t\equiv 1 \bmod \gp^2$. We have $c\equiv (bt)^2\bmod \gp^2$. Thus,
\[x^2+bx+c\equiv (x+bt)^2\bmod \gp^2.\]
We find the $\gp$-adic development of $f(x)$ is
\[f(x)=(x+bt)^2+(b-2bt)(x+bt)+b^2t^2-b^2t+c.\]
Since $b^2\equiv 4c\bmod \gp^2$, we have
\[b^2t^2-b^2t+c\equiv 4ct^2-4ct+c\equiv c-2c+c\equiv 0\bmod \gp^2.\]
Further, $b-2bt\equiv 0\bmod \gp^2$. Therefore, $\ind_{x+bt}(f)>0$, and Ore's theorem of the index (Theorem \ref{Thmofindex}) tells us that $f(x)$ is not $\gp$-maximal and hence not monogenic.

We have shown:
\begin{lemma}\label{Lem: NonMonoQuadp}
The irreducible polynomial $f(x)=x^2+bx+c$ is $\gp$-maximal for $\gp\nmid 2$ if and only if $\gp^2\dnd \Disc(f)=b^2-4c$.
\end{lemma}


Now suppose $\gl$ is a prime of $\Ocal_K$ with $\gl \mid 2$. If $\gl \nmid b$, then $f(x)$ is $\gl$-maximal. Conversely, if $\gl\mid b$, then we have $\gl^2\mid b^2$ and $\gl^2\mid \Disc(f)$. Further, letting $k_\gl=\Ocal_K/\gl$, write $|k_\gl|=2^h$. We have
\[x^2+bx+c\equiv x^2+c\equiv \left(x+c^{2^{h-1}}\right)^2 \bmod \gl.\]
Notice that if $f(x)$ is $\gl$-maximal, then Dedekind-Kummer factorization 
shows $\gl$ is totally ramified in $K\big(\sqrt{b^2-4c}\big)$. 

For convenience, write $c^{2^{h-1}}=\xi$. Taking the $(x+\xi)$-adic development,
\[f(x)=(x+\xi)^2+(b-2\xi)(x+\xi)-b\xi+\xi^2+c.\]
Since $\gl\mid b-2\xi$, we see that $f(x)$ is $\gl$-maximal if and only if $v_\gl(-b\xi+\xi^2+c)=1$. This can be rephrased as the following:
\begin{lemma}\label{Lem: NonMonoQuad2}
If $\gl\mid 2$, then $f(x)=x^2+bx+c$ is $\gl$-maximal if and only if either
$\gl\nmid b$ or
\[\gl\mid b \ \text{ and } \ 
c^{2^{h}}-bc^{2^{h-1}}+c \not\equiv 0 \bmod \gl^2, 
\ \text{ where } \ |k_\gl|=2^h.\]
In the latter case, $\gl$ is totally ramified in $K\big(\sqrt{b^2-4c}\big)$.
\end{lemma}



\subsection{The Main Theorem}
We restate our main theorem for the ease of the reader.\\

\noindent\textbf{Theorem \ref{Thm: Main}.} \textit{Let $f(x)=x^2+bx+c$ be an irreducible polynomial in $\ZZ[x]$. Let $\alpha_n$ be a root of the $n$-fold composition $f^n(x)$,
    and write $K_n=\QQ(\alpha_n)$. 
    The polynomial $f(x)$ is 2-maximal if and only if the composition $f^n(x)$ is 2-maximal for all $n$ if and only if one of the following hold:
    \begin{enumerate}
    \item[$\bullet$] $b$ is odd, 
    \item[$\bullet$] $b$ is even and $c\equiv 2 \bmod 4$,
    \item[$\bullet$] $b$ is even and $b+c\equiv 1 \bmod 4$. 
    \end{enumerate}
    In the first case (when $b$ is odd), 2 is unramified. In the latter two cases, 2 is totally ramified.
    Fix $N>0$ and suppose $f^n(x)$ is irreducible for $1\leq n\leq N$. For $k<N$, write $f_{N/k}(x)$ for the minimal polynomial of $\alpha_N$ over $K_k$. Let $p$ be an odd prime. The polynomial $f^n(x)$ is $p$-maximal for each $0<n\leq N$ if and only if 
        \[ p^2 \text{ does not divide } 4^{2^{n}}f^{n}\left(-\frac{b}{2}\right)=4^{2^{n-1}}f^{n-1}\left(-\frac{\Disc(f)}{4}\right)=\Norm_{K_n/\QQ}\left(\Disc(f_{n/n-1})\right)\] for each $0<n\leq N$. Here $f^0(x)$ is the identity function.}

The main results of \cite{AyadMcQuillan} are useful to establish irreducibility of iterates: Let $f(x)=x^2+bx+c\in\ZZ[x]$ and $\Disc(f)=b^2-c$. Suppose $f$ is irreducible. If $\Disc(f)\equiv 1 \bmod 4$, or if $\Disc(f)\equiv 0 \bmod 4$ but $\Disc(f)\not\equiv 0 \bmod 16$, then all iterates of $f$ are irreducible.\footnote{Every iterate being irreducible is called being \textit{stable} or \textit{dynamically irreducible} in the literature.} We check that if $f(x)$ is 2-maximal, then $f(x)$ satisfies a condition so that all iterates are irreducible: 
\begin{enumerate}
    \item[$\bullet$] $b$ odd $\implies$ $\Disc(f)\not\equiv 0\bmod 4$. $\checkmark$ 
    \item[$\bullet$] $b$ even and $c\equiv 2 \bmod 4$ $\implies$ $\Disc(f)\equiv b^2-8 \bmod 16.$ $\checkmark$
    \item[$\bullet$] $b\equiv 2 \bmod 4$ and $c \equiv 3 \bmod 4$ $\implies$ $\Disc(f)\equiv 4-12\equiv 8\bmod 16.$ $\checkmark$
    \item[$\bullet$] $b\equiv 0 \bmod 4$ and $c \equiv 1 \bmod 4$ $\implies$ $\Disc(f)\equiv 0-4\equiv 12\bmod 16.$ $\checkmark$
\end{enumerate}

\begin{proof}[Proof of 2-maximality.]
    First, note that Lemma \ref{Lem: NonMonoQuad2} gives the equivalence of the 2-maximality of $f(x)$ and one of the three conditions holding. Thus, we need only prove that each condition is sufficient for the 2-maximality of $f^n(x)$ for all $n\geq 1$.
    
    In the case where $b$ is odd, we again verify $2\dnd \Disc(f)=b^2-4c$. Thus $f(x)$ is 2-maximal and 2 is unramified in $K_1$. Suppose $\ZZ[\alpha_k]$ is $2$-maximal for all $k\leq n$ and 2 is unramified in $K_n$. The minimal polynomial of $\alpha_{n+1}$ over $K_n$ is $x^2+bx+c-\alpha_n$. The discriminant is $b^2-4(c-\alpha_n)$. Since no prime of $\Ocal_{K_n}$ above 2 divides this discriminant, we see $\Ocal_{K_n}[\alpha_{n+1}]$ is 2-maximal and 2 is unramified in $K_{n+1}$. Lemma \ref{Lem: MonoSteps} shows that $\ZZ[\alpha_{n+1}]$ is 2-maximal. Thus, we have shown that $f^n(x)$ is 2-maximal for every $n$.

    Suppose now that $b$ is even and $c\equiv 2 \bmod 4$. Thus $f(x)$ is 2-Eisenstein. A lemma of Odoni \cite[Lemma 1.3]{Odoni} shows that $f^n(x)$ is 2-Eisenstein for each $n$. Thus $\ZZ[\alpha_n]$ is 2-maximal and 2 is totally ramified in $K_n$ for every $n$. 

    We come to our most involved cases for 2-maximality: $b\equiv 2 \bmod 4$ and $c \equiv 3 \bmod 4$ or $b\equiv 0 \bmod 4$ and $c \equiv 1 \bmod 4$. As mentioned above, if either of these conditions are satisfied, Lemma \ref{Lem: NonMonoQuad2} shows $f(x)$ is 2-maximal. Further, the ideal $(2)$ is totally ramified. This supplies our base case. 
    
    Suppose that $b\equiv 2 \bmod 4$ and $c \equiv 3 \bmod 4$ or $b\equiv 0 \bmod 4$ and $c \equiv 1 \bmod 4$, that $f^k(x)$ is 2-maximal for all $k\leq n\geq 1$, and that $(2)=\gl^{2^n}$ in $\Ocal_{K_n}$. From the relative minimal polynomial of $\alpha_{n+1}$ and Lemma \ref{Lem: MonoSteps}, we see that $f^{n+1}(x)$ is 2-maximal if and only if $x^2+bx+c-\alpha_n$ is $\gl$-maximal. Thus, applying Lemma \ref{Lem: NonMonoQuad2}, we consider $(c-\alpha_n)^2-b(c-\alpha_n)+c-\alpha_n\bmod \gl^2$. Recalling $\gl^2\mid 2$,
    \begin{align*}
        (c-\alpha_n)^2-b(c-\alpha_n)+c-\alpha_n&\equiv c^2-2c\alpha_n+\alpha_n^2-bc+b\alpha_n+c-\alpha_n \\
        &\equiv c^2+c +\alpha_n^2-\alpha_n\\
        &\equiv \alpha_n^2-\alpha_n\\
        &\equiv -b\alpha_n-c+\alpha_{n-1}-\alpha_n\\
        &\equiv c+\alpha_{n-1}+\alpha_n\bmod \gl^2.
    \end{align*}
Taking norms and using the fact $\gl$ is totally ramified\footnote{See Lemma 3.3 of \cite{SmithDynRadical}.}, we see $c+\alpha_{n-1}+\alpha_n \equiv 0 \bmod \gl^2$ if and only if $\Norm_{K_n/K_{n-1}}\big(c+\alpha_{n-1}+\alpha_n\big)\equiv 0 \bmod \Norm_{K_n/K_{n-1}}(\gl)^2$. 

If $n=1$, then $\Norm_{K_n/K_{n-1}}(\gl)^2=4$ and $\alpha_{n-1}=0$. Using overlines to indicate conjugates,
\begin{align*}
    \Norm_{K_1/\QQ}\left(c+\alpha_1\right)&\equiv \left(c+\alpha_1\right)\left(c+\ol{\alpha_1}\right)\\
    &\equiv c^2+c\left(\alpha_1+\ol{\alpha_1}\right)+\alpha_1\ol{\alpha_1}\\
    &\equiv c^2-bc+c \bmod 4.
\end{align*}
Conveniently, this is simply our base case. As $c^2-bc+c\not\equiv 0 \bmod 4$, we see $f^2(x)$ is 2-maximal and $\gl$ is totally ramified in $K_{2}$.

If $n>1$, then we have $\Norm_{K_n/K_{n-1}}(\gl)^2 \mid 2$. We compute,  
\begin{align*}
    \Norm_{K_n/K_{n-1}}\left(c+\alpha_{n-1}+\alpha_n\right)&\equiv \left(c+\alpha_{n-1}+\alpha_n\right)\left(c+\alpha_{n-1}+\ol{\alpha_n}\right)\\
    &\equiv c^2 + \alpha_{n-1}^2 +\alpha_{n-1}\left(\alpha_n+\ol{\alpha_n}\right) + c\left(\alpha_n+\ol{\alpha_n}\right)+\alpha_n\ol{\alpha_n}\\
    &\equiv c^2 + \alpha_{n-1}^2-\alpha_{n-1}b - cb+c - \alpha_{n-1}\\
    &\equiv \alpha_{n-1}^2 - \alpha_{n-1} \bmod \Norm_{K_n/K_{n-1}}(\gl)^2.\\
\end{align*}
Our induction hypothesis shows $\alpha_{n-1}^2 - \alpha_{n-1}\not\equiv 0 \bmod \Norm_{K_n/K_{n-1}}(\gl)^2$ since $f^n(x)$ is $2$-maximal. Thus $f^{n+1}(x)$ is 2-maximal and $\gl$ (hence also 2) is totally ramified in $K_{n+1}$. 
\end{proof}

\begin{proof}[Proof of $p$-maximality]
    We proceed by induction. Let $p$ be an odd prime. Lemma \ref{Lem: NonMonoQuadp} supplies our base case: $f(x)$ is $p$-maximal if and only if $p^2\dnd b^2-4c=\Disc(f)$. 

    Now suppose that $f^k(x)$ is $p$-maximal and $p^2\dnd \Norm_{K_k/\QQ}\big(\Disc(f_{k+1/k})\big)$ for all $k\leq n$. We see $f^{n+1}(x)$ is $p$-maximal if and only if $f_{n+1/n}(x)=x^2+bx+c-\alpha_n$ is $\gp$-maximal for each prime $\gp$ of $\Ocal_{K_N}$ above $p$. For a given $\gp$, Lemma \ref{Lem: NonMonoQuadp} shows this occurs if and only if $\gp^2\dnd b^2-4(c-\alpha_n)=\Disc(f_{n+1/n})$. 
    
    Taking norms, it is clear that  $\gp^2 \mid \Disc(f_{n+1/n})\implies p^2\mid \Norm_{K_n/\QQ}\big(\Disc(f_{n+1/n})\big)$. A computation with a complete factorization of $f^n(x)$ shows that 
    \[\Norm_{K_n/\QQ}\big(\Disc(f_{n+1/n})\big)=4^{2^n}f^n\left(-\frac{\Disc(f)}{4}\right).\] 
    
    For the converse, suppose $p^2\mid \Norm_{K_n/\QQ}\big(\Disc(f_{n+1/n})\big)$. 
    As 
    \[f(x)=x^2+bx+c=\left(x+\frac{b}{2}\right)^2-\frac{b^2}{4}+c,\] we see that $f^{n+1}(x)$ is a polynomial in $\ZZ_p\big[(x+\frac{b}{2})^2\big]$. Making the change of variables $X=x+\frac{b}{2}$, we have $f^{n+1}(X)\in \ZZ_p[X^2]$. Since $\frac{b}{2}\in \ZZ_p$, we see that this change of variables preserves the $p$-maximality or failure of $p$-maximality of $f^{n+1}$. The constant coefficient of $f^{n+1}(X)$ is $f^{n+1}(0)=f^n\big(-\frac{b^2}{4}+c\big)$. Since $f^{n+1}(X)\in \ZZ_p[X^2]$, our supposition that $p^2$ divides $f^n\big(-\frac{b^2}{4}+c\big)$ shows that $f^{n+1}(X)$ has a multiple root modulo $p$. Developing with respect to $X$, we have that $v_p\big(f^{n+1}(X=0)\big)\geq 2$. Using Theorem \ref{Thmofindex}, the lattice points on or below the principal $X$-polygon show that $f^{n+1}(X)$, and hence $f^{n+1}(x)$, is not $p$-maximal. Induction yields the result.
\end{proof}



We shift our attention to establishing a few results that were used in the proof above.

\begin{lemma}\label{Lem: MonoSteps}
Let $f(x)=x^2+bx+c$ and keep the notation as above. Suppose $f^N(x)$ is monogenic for some $N>0$. If $\Ocal_{K_N}[\alpha_{N+1}]=\Ocal_{K_{N+1}}$, then $\ZZ[\alpha_{N+1}]=\Ocal_{K_{N+1}}$.
\end{lemma}

\begin{proof}
    Note that $\alpha_{N+1}$ satisfies $x^2+bx+c-\alpha_N$, so $\alpha_{N+1}^2+b\alpha_{N+1}+c=\alpha_N$. Thus $\ZZ[\alpha_N]\subset \ZZ[\alpha_{N+1}]$. Since $\Ocal_{K_N}=\ZZ[\alpha_N]$, we have $\Ocal_{K_N}[\alpha_{N+1}]=\ZZ[\alpha_N][\alpha_{N+1}]=\ZZ[\alpha_{N+1}]$.
\end{proof}



\begin{proposition}\label{Prop: ContinuedNotpMax}
    If $p\mid \big[\Ocal_{K_n}:\ZZ[\alpha_n]\big]$, then $p\mid \big[\Ocal_{K_m}:\ZZ[\alpha_m]\big]$ for all $m>n$. In particular, $\ZZ[\alpha_m]\cap \Ocal_{K_n}$ is not $p$-maximal for any $m>n$.
\end{proposition}

\begin{proof}
    Suppose $p\mid \big[\Ocal_{K_n}:\ZZ[\alpha_n]\big]$. Thus, there is an algebraic integer of the form
    \[\gamma=\frac{a_{2^n-1}\alpha_n^{2^n-1}+\cdots +a_1\alpha_n+a_0}{p},\]
    with each $a_i\in\ZZ$ and not all divisible by $p$. We can assume $a_{2^n-1}$ is not divisible by $p$ by subtracting $\frac{a_{2^n-1}\alpha_n^{2^n-1}}{p}$ and continuing to subtract terms that are divisible by $p$ until a leading coefficient that is not divisible by $p$ is obtained. 

    The minimal polynomial for $\alpha_m$ over $K_n$ is $f^{m-n}(x)=\alpha_n$. Further, $f^{m-n}(x)-\alpha_n$ is monic and irreducible of degree $2^{m-n}$. We have 
    \[\gamma=\frac{a_{2^n-1}f^{m-n}\big(\alpha_m\big)^{2^n-1}+\cdots +a_1f^{m-n}\big(\alpha_m\big)+a_0}{p}.\]
    Expanding, 
    \[\gamma=\frac{a_{2^n-1}\alpha_m^{2^m-2^{m-n}}+\text{lower order terms}}{p}.\]
    Since $\{\alpha_m^i\}_{i=0}^{2^m-1}$ yields a $\QQ$-basis for $K_m$, this representation of $\gamma$ is unique. Hence $\gamma\notin \ZZ[\alpha_m]$, and $\ZZ[\alpha_m]\cap \Ocal_{K_n}$ is not $p$-maximal.
\end{proof}

As a porism, we have the following:

\begin{proposition}
    Suppose $\ZZ[\beta]$ is not $p$-maximal. If $\beta=g(\delta)$ for some monic $g(x)\in \ZZ[x]$ and some algebraic integer $\delta$, then $\ZZ[\delta]$ is not $p$-maximal.
\end{proposition}

Once $\ZZ[\alpha_n]$ fails to be $p$-maximal, we can obtain a lower bound on the powers of $p$ dividing the index of further iterates.

\begin{proposition}\label{Prop: NonMonoExt}
    If $p\mid \big[\Ocal_{K_n}:\ZZ[\alpha_n]\big]$, then $p^{2^{(m-n)}}\mid \big[\Ocal_{K_m}:\ZZ[\alpha_m]\big]$ for all $m>n$. More generally, suppose $g(x)\in \ZZ[x]$ is monic of degree $d$ and such that the first $m$ iterates are irreducible. Let $\beta_n$ be a root of $g^n(x)$ and write $L_n=\QQ(\beta_n)$. If $p^k\mid \big[\Ocal_{L_n}:\ZZ[\beta_n]\big]$, then $\big(p^k\big)^{d^{m-n}}$ divides $\big[\Ocal_{L_m}:\ZZ[\beta_m]\big]$.
\end{proposition}


\begin{proof}
Suppose $p^k\mid \big[\Ocal_{L_n}:\ZZ[\beta_n]\big]$, and consider $\Ocal_{L_n}[\beta_m]$. We see that $\big[\Ocal_{L_n}[\beta_m]:\ZZ[\beta_m]\big]$ divides $\big[\Ocal_{L_m}:\ZZ[\beta_m]\big]$. As $\ZZ[\beta_n][\beta_m]=\ZZ[\beta_m]$, we compute
\begin{align*}
    \big[\Ocal_{L_n}[\beta_m]:\ZZ[\beta_m]\big]&= \big| \Ocal_{L_n}[\beta_m]/\ZZ[\beta_n][\beta_m]\big|\\
    &= \left| \big(\Ocal_{L_n} \oplus \cdots \oplus \Ocal_{L_n}\beta_m^{d^{m-n}-1}\big)\big/ \big( \ZZ[\beta_n] \oplus \cdots \oplus \ZZ[\beta_n]\beta_m^{d^{m-n}-1}\big)\right|\\
    &= \left| \left(\Ocal_{L_n}/\ZZ[\beta_n]\right)\oplus \left(\Ocal_{L_n}/\ZZ[\beta_n]\right)\beta_m\oplus \cdots \oplus \left(\Ocal_{L_n}/\ZZ[\beta_n]\right)\beta_m^{d^{m-n}-1}\right|\\
    &=\big|\Ocal_{L_n}/\ZZ[\beta_n]\big|^{d^{m-n}}.
\end{align*}
The result follows.
\end{proof}

The contrapositive yields
\begin{corollary}\label{Cor: MonoSubfields}
    If $g^n(x)$ is monogenic, then $g^i(x)$ is monogenic for all $i<n$.
\end{corollary}

\section{Prime factorizations of 2}\label{Sec: Factorization}
In this section we describe the factorization of $(2)$ in $K_n$ when $f(x)$ is $2$-maximal. 
Theorem \ref{Thm: Main} fully describes the cases where 2 is ramified, but we give a quick summary:

\subsection{Ramified}\label{Sec: RamifiedSplitting}
Note that in order for $(2)$ to be ramified it is necessary and sufficient that $2\mid b$. When $2\mid b$ and $c\equiv 2 \pmod 4$, $f(x)$ is $2$-Eisenstein and so is $f^n(x)$. Thus $2$ does not divide $\big[\Ocal_{K_n}:\ZZ[\alpha_n]\big]$. Hence, since $f^n(x) = x^{2^n}$ in $\FF_2[x],$ we have $2\Ocal_{K_n}=\gp^{2^n}$, with $\gp=(2,\alpha_n)$.

When $2\mid b$ and $c$ is odd, then $f(x)$ is $2$-maximal if and only if $b+c\equiv 1 \pmod 4$. We see that $f^{2k}(x)$ and $f^{2k+1}(x+1)$ are $2$-Eisenstein, so $2\Ocal_{K_n}=\gp^{2^n}$ with $\gp=(2,\alpha_n)$ when $n$ is even and $\gp=(2,\alpha_n+1)$ when $n$ is odd. 

In the remaining ramified cases, $f^n(x)$ is not $2$-maximal, so the factorization of $(2)$ does not correspond to factorization of $f^n(x)$ in $\FF_2[x]$.  

\subsection{Unramified}
When $(2)$ is unramified, 
$\ZZ[\alpha_n]$ is $2$-maximal for every $n$, so Dedekind-Kummer factorization shows the prime ideal factorization of $(2)$ corresponds to the factorization of $f^n(x)$ into irreducibles in $\FF_2[x]$. 

We denote $F(x)=x^2+x+1$ and $G(x)=x^2+x$ in $\FF_2[x]$ for the two quadratics with non-zero $x$ coefficient.

\begin{lemma}\label{Lem: RedFactors}
For every $n\geq 1$, $G^{n+1}(x)=G^n(x)F^n(x)$.
\end{lemma}
\begin{proof}
We have 
\[G^2(x) = (x^2+x)^2+x^2+x=(x^2+x)(x^2+x+1)=G(x)F(x),\]
and 
\[G^{n+1}(x)=\big(G^n(x)\big)^2+G^n(x)=G^n(x)\left(G^n(x)+1\right).\]
We would like to show that $F^n(x)=G^n(x)+1$, and we proceed by induction. When $n=1$, it follows from the definitions, so suppose $F^k(x)=G^k(x)+1$ for some $k\geq 1$. 

Then $F^{k+1}(x)+1=F^k(x)^2+F^k(x)=G\big(F^k(x)\big)$, which is $G(G^k(x)+1)$ by induction, and we would like to show this equals $G^{k+1}(x)$. For every polynomial $P(x)$ in $\FF_2[x]$, 
\[
G(P(x)+1)=(P(x)+1)^2+P(x)+1=P(x)^2+P(x)=G\big(P(x)\big),
\]
which completes the proof. 
\end{proof}

This shows that the factorization of $G^{n}(x)$ is determined by the factorizations of $F^k(x)$ for $1\leq k \leq n-1$, specifically:
\[G^n(x)=x(x+1)F(x)F^2(x)\cdots F^{n-1}(x).\]
To determine the factorization of $2\Ocal_K$ in either case, it is enough to factor each $F^k(x).$

\begin{lemma}\label{Lem: IrredFactors}
For every $n$, $F^n(x)$ is a product of distinct irreducible polynomials in $\FF_2[x]$ of degree $2^m$, where $2^{m-1}\leq n<2^m$.
\end{lemma}
\begin{proof}
To prove the lemma, we start by showing that $F^n(x)$ divides $x^{2^{2^m}}-x$ when $n<2^m$. Using the fact that an irreducible polynomial of degree $d$ divides $x^{2^r}-x$ if and only if $d$ divides $r$ \cite[2.1.29]{Handbook}, the degree of every irreducible factor of $F^n(x)$ is a power of $2$ and at most $2^m$. 

From Lemma \ref{Lem: RedFactors}, we have $F^n(x)\mid G^{2^m}(x)$ whenever $n<2^m$. We will show that $G^{2^m}=x^{2^{2^m}}-x$ by induction. We have $G(x)=x^2-x$, and 
    \begin{align*}
        G^{2^m}(x)&=G^{2^{m-1}}\left(G^{2^{m-1}}(x)\right) \\
        &=\left(x^{2^{2^{m-1}}}-x\right)^{2^{2^{m-1}}}-\left(x^{2^{2^{m-1}}}-x\right) \\
        &=x^{2^{2^m}}-x.
    \end{align*}

Next, we show that $F^i(x)$ and $F^j(x)$ are coprime when $i\neq j$ by showing that $F^i \mid F^j+1$ whenever $i<j$. In the proof of Lemma \ref{Lem: RedFactors}, we showed $F^j(x)+1=G^j(x),$ and if $i<j$, $F^i(x)$ is a factor of $G^j(x)$. Hence, there are no irreducible factors shared by any distinct iterates of $F(x)$. 
We complete the proof by counting the number of irreducible polynomials of degree $2^k$ for any $k\leq m$ to show that every irreducible polynomial with degree $2^k$ is a factor of exactly one $F^n(x)$.

Working inductively, let $m\geq 2$ and suppose that for each $1\leq k < m$ every irreducible of degree $2^k$ in $\FF_2[x]$ is a factor of some $F^n(x)$ with $n<2^{m-1}$. This is true when $m=2$, since there is only one irreducible quadratic in $\FF_2[x]$, and it is $F(x)$. We would like to show that every irreducible polynomial with degree $2^m$ is a factor of some $F^n(x)$ with $n<2^m$. If $n<2^{m-1},$ then each irreducible factor of $F^n(x)$ has degree less than $2^{m-1}$, so consider the product

\[
\prod_{n=2^{m-1}}^{2^m-1}F^n(x).
\]

Each of its irreducible factors has degree equal to a power of $2$ and less than $2^m$. By assumption every irreducible polynomial with degree $2^k$ for $k<m$ is a factor of some $F^n(x)$ with $n<2^{m-1}$. Since $F^i(x)$ and $F^j(x)$ are coprime, those polynomials with degree less than $2^m$ cannot appear in the product, so every irreducible factor of the product must have degree $2^m$. 

The product has degree 
$$\sum_{n=2^{m-1}}^{2^m-1}2^n=2^{2^m}-2^{2^{m-1}},$$ and there are $(2^{2^m}-2^{2^{m-1}})/2^m$ irreducible polynomials of degree $2^m$ \cite[2.1.24]{Handbook}, so each appears exactly once in the product.

Thus $F^n(x)$ factors into $2^{n-m}$ irreducible polynomials of degree $2^m$, where $2^m$ is the smallest power of 2 larger than $n$, that is, $2^{m-1}\leq n < 2^m$.  
\end{proof}

Dedekind-Kummer factorization and the lemmas above yield the following:
\begin{theorem}\label{Thm: UnramFactors}
Let $f(x) = x^2+bx+c$ be an irreducible polynomial in $\ZZ[x]$ with $2\dnd b$, $f^{n}(x)$ the $n$-fold iterate of $f(x)$, $K_n$ the number field generated by $f^{n}(x)$, and $\Ocal_{K_n}$ its ring of integers. 

When $2\dnd c$, let $2^m$ denote the smallest power of $2$ greater than $n$, $2^{m-1}\leq n <2^m$. Then   
\[2\Ocal_{K_n}=\prod_{i=1}^{2^{n-m}}\gp_i,\]
where the residue class degree of each $\gp_i$ is $[\Ocal_{K_n}/\gp_i : \FF_2]=2^m$. 

When $2\mid c$, let $2^m$ denote the smallest power of $2$ greater than $n-1$, so $2^{m-1}< n \leq 2^m$. Then for $n>1$, 
\begin{equation}\label{Eq: 2splitting}
2\Ocal_{K_n}=\gp_{0,1}\gp_{0,2} \prod_{r=1}^{m-1} \left(\prod_{i=1}^{2^{2^r-r}-2^{2^{r-1}-r}}\gp_{r,i}\right)\prod_{j=2^{m-1}}^{n-1}\left(\prod_{i=1}^{2^{j-m}}\gp_{m,i}\right),
\end{equation}
where the residue class degree of $\gp_{r,i}$ is $\big[\Ocal_{K_n}/\gp_{r,i} : \FF_2\big]=2^r$.

\end{theorem}

The first collection of products of ideals in \eqref{Eq: 2splitting} corresponds to all irreducible polynomials with degrees $2^r$ for $1\leq r \leq m-1$, while the remaining products correspond to a subset of the irreducibles of degree $2^m$, or all of them when $n=2^m$.


\subsection{Examples}
Let $g(x)=x^2-x+2$. Note that $g(x)$ is dynamically irreducible and $2$ is unramified. Theorem \ref{Thm: UnramFactors} implies that $2\Ocal_{K_5}$ splits into two factors of inertia degree $1$, one of degree $2$, three of degree $4$, and two of degree $8$. Those of degree less than $8$ correspond to all irreducible polynomials of degree $2^r$ in $\FF_2[x]$, where $0\leq r\leq 2$. A calculation with SageMath\cite{sagemath} gives



\begin{align*}
(2) = & (2, \alpha_5)(2, \alpha_5 + 1)(2, \alpha_5^2 + \alpha_5 + 1) \\
& (2, \alpha_5^4 + \alpha_5^3 + \alpha_5^2 + \alpha_5 + 1)(2, \alpha_5^4 + \alpha_5^3 +1)(2, \alpha_5^4 + \alpha_5 + 1) \\
& (2, \alpha_5^8 + \alpha_5^6 +\alpha_5^5 + \alpha_5^4 +\alpha_5^3 +\alpha_5 + 1)\\
& (2, \alpha_5^8 + \alpha_5^6 +\alpha_5^5 +\alpha_5^3 + 1).
\end{align*}

Let $f(x)=x^2-x+1$. Again, we have dynamic irreducibility and $2$ is unramified. As $c$ is odd, $2\Ocal_{K_5}$ splits into factors of equal inertia degree. Since $2^2\leq 5 < 2^3$, each factor has degree $8$, and thus there are four factors. Again, using SageMath, we have:


\begin{align*}
(2)=&(2, \alpha_5^8 + \alpha_5^6 + \alpha_5^5 + \alpha_5^2 + 1)\\
&(2, \alpha_5^8 + \alpha_5^6 + \alpha_5^5 + \alpha_5 + 1)\\
&(2, \alpha_5^8 +\alpha_5^5 + \alpha_5^3 + \alpha_5 + 1)\\
&(2, \alpha_5^8 + \alpha_5^5 + \alpha_5^4 +\alpha_5^3 + \alpha_5^2 + \alpha_5 + 1).
\end{align*}

Note that we do not need to know if $g^n(x)$ or $f^n(x)$ are monogenic to explicitly determine $2\Ocal_{K_n}$. As $2$ is unramified, it does not divide the index $\big[\Ocal_{K_n}:\ZZ[\alpha_n]\big]$. 

When $n=2^m-1$, the factors of $F^n(x)$ correspond to all polynomials with non-zero trace, that is, all normal\footnote{A polynomial is \textit{normal} if the roots generate a Galois-invariant basis.} polynomials over $\FF_2$ of degree $2^m$ by another counting argument. (See \cite[5.2.9]{Handbook} for the correspondence between non-zero trace and normality and \cite[5.2.12]{Handbook} for the matching count.) The factors of $F^{2^m-1}(x)$ being every irreducible with the right degree and certain coefficients prompts the following question: 
{\center 
    Is it true that the irreducible factors of $F^{2^{m-1}}(x)$ are the irreducibles $\sum_{i=0}^{2^m} a_ix^i$ with \[a_{2^m-1}=0 \text{ and } a_{2^m-1-2^i}=1 \text{ for all } 0\leq i < m?\]
}

\section{Application: Dynamically Monogenic PCF Quadratic Polynomials}\label{Sec: Examples}
In this section, we construct three infinite families of iterated quadratic towers of monogenic fields. These families come from the three families of \emph{post-critically finite}, monic, quadratic, integer polynomials characterized in \cite{Goksel}. We were initially motivated by \cite{Castillo}, where Castillo considers $f(x)=x^2-2$. One has
\begin{example}\label{Ex. p=2}
	\begin{align*}
		f^1(x) &= x^2-2,\\
		f^2(x) &= x^{4} - 4x^{2} + 2,\\
		f^3(x) &= x^{8} - 8x^{6} + 20x^{4} - 16x^2 +2,\\
		f^4(x) &= x^{16} - 16 x^{14} + 104x^{12} - 352x^{10} + 660x^8 - 672x^6 + 336x^4 -64x^2 + 2,\\
		f^5(x) &= x^{32} - 32 x^{30} + \cdots + 5440 x^{4} - 256 x^{2} + 2.
	\end{align*}
	We can see that $f^n(0)=\big(f^{n-1}(0)\big)^2-2$, so that $f^n(0)=2$ for all $n>1$. Note that $0$ is the unique critical point of $x^2-2$, and that its forward orbit is a finite set. Thus, applying Theorem \ref{Thm: Main}, for example, every iterate of $f(x)$ is monogenic! Castillo shows this and motivates the following question.

 \begin{question}\label{Q: AlwaysMono} What are the other $f(x)$ such that $f^n(x)$ is monogenic for all $n\geq 1$?
\end{question}
\end{example}


In investigating this question, we found \cite{AitkenIterated}, where the authors analyze ramification in iterated extensions. In Section 6, they conduct a deep investigation of $f(x)=x^2-2$. Indeed, they show that $f^n(x)-t_0$ is monogenic for $n\geq 1$ when $t_0\equiv 0,1\bmod 4$ and $t_0\pm2$ are both squarefree, which is the third family in Proposition \ref{Prop: MonoPCFs}. Another example of dynamically monogenic PCF polynomials can be found in \cite{GassertChebyshev}, where Gassert gives conditions under which prime-power Chebyshev polynomials are monogenic.

A polynomial is \emph{post-critically finite}, or PCF, if the forward orbit of its critical points under iteration, the \emph{critical orbit}, is a finite set. Thus, if $-D_f/4$ is a fixed point of $f^i(x)$, then $f(x)$ is PCF. We use the notation and results of \cite{Goksel}.

\begin{proposition}\label{Prop: MonoPCFs}
    The following PCF polynomials yield families in which every iterate is monogenic:
    \begin{align*}
	f_a(x) &= (x+a)^2-a   = x^2+2ax+a^2-a ,  &\, D/4&=a,\\
	g_a(x) &= (x+a)^2-a-1 = x^2+2ax+a^2-a-1, &\, D/4&=a+1, \\
        h_a(x) &=(x+a)^2-a-2  = x^2+2ax+a^2-a-2, &\, D/4&=a+2.
\end{align*}
Explicitly, every iterate of $f_a(x)$ and $g_a(x)$ is monogenic if and only if $a$ is squarefree and $D/4$ is squarefree and congruent to $2$ or $3$ modulo $4$. Every iterate of $h_a(x)$ is monogenic if and only if $a-2$ and $D/4=a+2$ are squarefree and $D/4$ is congruent to $2$ or $3$ modulo 4. 
\end{proposition}

\begin{remark}\label{Rmk: Specializations}
    Notice that $f^n_a(x-a)=p^n(x)-a$, where $p(x)=x^2$. Further, $g_a^n(x-a)=q^n(x)-a$ and $h_a^n(x-a)=r^n(x)-a$, where $q(x)=x^2-1$ and $r(x)=x^2-2$, respectively. Thus Proposition \ref{Prop: MonoPCFs} gives necessary and sufficient conditions for the specializations of iterates of the three families $x^2$, $x^2-1$, and $x^2-2$ to be monogenic.
\end{remark}

The proof is an application of Theorem \ref{Thm: Main}. Note that in all cases 2 is totally ramified.

\begin{proof}
    For $f_a(x)$, Theorem \ref{Thm: Main} shows that every iterate is 2-maximal if and only if $a\equiv 2$ or 3 modulo 4. Thus, we see the necessity of this condition. This also implies that every iterate of $f(x)$ is irreducible. Now, if $p$ is an odd prime, we see $f^n\big(-\frac{b}{2}\big)=f^n(-a)=-a=-D/4.$ Thus, for $p$-maximality, it is necessary and sufficient that $D/4=a$ is not divisible by $p^2$. Hence, combining with our constraints for the prime 2, we have our necessary and sufficient conditions.

    For $g_a(x)$, Theorem \ref{Thm: Main} shows that every iterate is 2-maximal if and only if $a\equiv 1$ or 2 modulo 4. Thus, we see the necessity of the condition that $D/4=a+1\equiv 2,3 \bmod 4$. This also implies that every iterate of $g(x)$ is irreducible. Now, if $p$ is an odd prime, we see $g\big(-\frac{b}{2}\big)=g(-a)=-a-1$ and $g(-a-1)=-a$. Thus, for $p$-maximality, it is necessary and sufficient that $D/4=a+1$ and $a$ are both not divisible by $p^2$. Hence, combining with our constraints for the prime 2, we have our necessary and sufficient conditions.

    For $h_a(x)$, Theorem \ref{Thm: Main} shows that every iterate is 2-maximal if and only if $a\equiv 0$ or 1 modulo 4. Thus, we see the necessity of the condition that $D/4=a+2\equiv 2,3 \bmod 4$. This also implies that every iterate of $h(x)$ is irreducible. Now, if $p$ is an odd prime, we see $h\big(-\frac{b}{2}\big)=h(-a)=-a-2$, $h(-a-2)=-a+2$, and $h(-a+2)=-a+2$. Thus, for $p$-maximality, it is necessary and sufficient that $D/4=a+2$ and $a-2$ are both not divisible by $p^2$. Hence, combining with our constraints for the prime 2, we have our necessary and sufficient conditions.
\end{proof}

Note that in each family, there are infinitely many integers satisfying the squarefree and congruence conditions. These three families of polynomials also have finite ramification sets, 
since the discriminant only depends on the orbit of the critical points.

One can see that the critical orbit is key to understanding monogenicity. Forthcoming work of the authors and Joachim K\"onig gives general criteria for the monogenicity of arbitrary integer polynomials in terms of the critical points and uses this to construct a variety of families of dynamically monogenic polynomials.

\section*{Acknowledgments} The authors would like to thank Istv\'an Ga\'al for hosting the conference that inspired this collaboration.

\bibliography{Bibliography}
\bibliographystyle{alpha}

\end{document}